\documentclass[12pt]{amsart}
\usepackage{amsmath,amssymb}
\usepackage{pslatex,mathptmx}
\usepackage[all,ps]{xypic}
\usepackage[utf8]{inputenc}

\newtheorem{theorem}{Theorem}[section]
\newtheorem{corollary}[theorem]{Corollary}
\newtheorem{lemma}[theorem]{Lemma}
\newtheorem{proposition}[theorem]{Proposition}

\theoremstyle{definition}
\newtheorem{definition}[theorem]{Definition}
\newtheorem{remark}[theorem]{Remark}
\newtheorem{example}[theorem]{Example}
\newtheorem{conjecture}[theorem]{Conjecture}

\newcommand{\Hom}{{\rm Hom}}
\newcommand{\Tor}{{\rm Tor}}
\newcommand{\Ext}{{\rm Ext}}
\newcommand{\CB}{{\rm CB}}
\newcommand{\CH}{{\rm CH}}
\newcommand{\colim}{{\rm colim}}
\renewcommand{\geq}{\geqslant}
\renewcommand{\leq}{\leqslant}

\newcommand{\xra}[1]{\xrightarrow{\ {#1}\ }}
\newcommand{\B}[1]{\mathbb{#1}}
\newcommand{\C}[1]{\mathcal{#1}}
\newcommand{\rmod}[1]{\text{\rm {\bf mod}-}{#1}}
\newcommand{\lmod}[1]{{#1}\text{\rm -{\bf mod}}}

\numberwithin{equation}{section}
\title{Jacobi-Zariski Exact Sequence for Hochschild Homology and Cyclic (Co)Homology}
\author{Atabey Kaygun}
\email{kaygun@itu.edu.tr}
\address{Department of Mathematics, Istanbul Technical University,  Istanbul, Turkey}

\begin{document}


\begin{abstract}
  We prove that for an inclusion of unital associative but not
  necessarily commutative $\Bbbk$-algebras $\C{B}\subseteq \C{A}$ we
  have long exact sequences in Hochschild homology and cyclic
  (co)homology akin to the Jacobi-Zariski sequence in Andr\'e-Quillen
  homology, provided that the quotient $\C{B}$-module $\C{A}/\C{B}$ is
  flat.  We also prove that for an arbitrary r-flat morphism
  $\varphi\colon\C{B}\to\C{A}$ with an H-unital kernel, one can
  express the Wodzicki excision sequence and our Jacobi-Zariski
  sequence in Hochschild homology and cyclic (co)homology as a single
  long exact sequence.
\end{abstract}

\maketitle

\section*{Introduction}

Let $\Bbbk$ be a ground field.  Assume we have an inclusion of
associative commutative unital $\Bbbk$-algebras $B\subseteq A$.  Then
for any $A$-bimodule $N$, one obtains a long exact sequence in
Andr\'e-Quillen homology~\cite[Thm.5.1]{Quillen:AQCohomology}
\[\cdots\to D_{n+1}(A|B;N)\to D_n(B|\Bbbk;N)\to 
D_n(A|\Bbbk;N) \to D_n(A|B;N)\to \cdots \] which is often referred as
the Jacobi-Zariski long exact
sequence~\cite[Sect.3.5]{Loday:CyclicHomology}.  In this paper we show
that there are analogous long exact sequences for ordinary
(co)homology, Hochschild homology and cyclic (co)homology of
$\Bbbk$-algebras of the form (written here for Hochschild homology)
\[ \cdots \to HH_{n+1}(\C{A}|\C{B})\to HH_n(\C{B}|\Bbbk)\to 
  HH_n(\C{A}|\Bbbk)\to HH_n(\C{A}|\C{B})\to\cdots 
\]
for $n\geq 1$.  We prove the existence under the condition that we
have an unital associative (not necessarily commutative) algebra
$\C{A}$ and a unital subalgebra $\C{B}$ such that the quotient
$\C{B}$-module $\C{A}/\C{B}$ is flat.  In the sequel, such inclusions
$\C{B}\subseteq \C{A}$ of unital $\Bbbk$-algebras are called {\em
  reduced-flat} ({\em r-flat} in short) extensions.  The condition of
r-flatness is slightly more restrictive than $\C{A}$ being flat over
$\C{B}$ but there are plenty of relevant examples. (See
Subsection~\ref{RFlat})

There are similar long exact sequences in the literature for other
cohomology theories of $\Bbbk$-algebras.  The relevant sequence we
consider here is the Wodzicki excision
sequence~\cite{Wodzicki:Excision} for Hochschild homology and cyclic
(co)homology (written here for Hochschild homology)
\[ \cdots\to HH_n(\C{I})\to HH_n(\C{B})\to HH_n(\C{A})\to
HH_{n-1}(\C{I}) \cdots \] for an epimorphism $\pi\colon\C{B}\to \C{A}$
of unital $\Bbbk$-algebras with an H-unital kernel $\C{I}:=ker(\pi)$
(Subsection~\ref{HUnital}).  The Wodzicki excision sequence
characterizes homotopy cofiber of the morphism of differential graded
$\Bbbk$-modules $\pi_*\colon \CH_*(\C{B})\to\CH_*(\C{A})$ induced by
$\pi$ as the suspended Hochschild complex $\Sigma\CH_*(\C{I})$ of the
ideal $\C{I}$.  Our Jacobi-Zariski sequence, on the other hand,
characterizes the same homotopy cofiber as the relative Hochschild
chain complex $\CH_*(\C{A}|\C{B})$ (relative \`a la Hochschild
\cite{Hochschild:RelativeHomology}) for a monomorphism $\C{B}\to
\C{A}$ of $\Bbbk$-algebras.

Now, assume $\varphi\colon \C{B}\to\C{A}$ is an arbitrary morphism of
unital $\Bbbk$-algebras such that $\C{I}:=ker(\varphi)$ is H-unital
and the quotient $\C{B}$-module $\C{A}/im(\varphi)$ is flat.  Under
these conditions, we show in Theorem~\ref{AllInOne} (written here for
Hochschild chain complexes) that homotopy cofiber $\CH_*(\C{A},\C{B})$
of the morphism $\varphi_*\colon \CH_*(\C{B})\to\CH_*(\C{A})$ induced
by $\varphi$ fits into a homotopy cofibration sequence of the form
\[ \Sigma\CH_{*\geq 1}(\C{I})\to \CH_{*\geq 1}(\C{A},\C{B})\to \CH_{*\geq 1}(\C{A}|\C{B}) \]
which gives us an appropriate long exact sequence.  As one can
immediately see, we get Wodzicki's characterization of the homotopy
cofiber when $\varphi$ is an epimorphism and our Jacobi-Zariski
characterization when $\varphi$ is a monomorphism.

\subsection*{Overview} In Section~\ref{Preliminaries} we review some
standard constructions and facts that we are going to need in the
course of proving our main result, mostly in order to establish
notation. Then in Section~\ref{JZT} we prove the existence of the
Jacobi-Zariski long-exact sequence for ordinary homology of algebras.
In Section~\ref{JZE} we gradually develop the same result for
cohomology under certain restrictions on the dimension of the algebra
then we remove those restrictions and place them on the coefficient
modules.  Finally in Section~\ref{MainResult}, we prove the existence
of the Jacobi-Zariski long exact sequence for Hochschild homology and
cyclic (co)homology of associative unital algebras, and then construct
a homotopy cofibration sequence extending both the Jacobi-Zariski
sequence and Wodzicki excision sequence in Hochschild homology and
cyclic (co)homology.

\subsection*{Standing assumptions and conventions}
We use $\Bbbk$ to denote our ground field.  We make no assumptions on
the characteristic of $\Bbbk$.  All unadorned tensor products
$\otimes$ are assumed to be over $\Bbbk$.  We will use
$\bigoplus_i X_i$ and $\sum_i X_i$ to denote respectively the external
and the internal sum of a collection of subspaces $\{X_i\}_{i\in I}$
of a $\Bbbk$-vector space $V$.  The $\Bbbk$-algebras we consider are
all unital and associative but not necessarily commutative. We make no
assumptions on the $\Bbbk$-dimensions of these algebras unless
otherwise is explicitly stated.  We will use $\C{A}^e$ to denote the
enveloping algebra $\C{A}\otimes \C{A}^{op}$ of an associative unital
algebra $\C{A}$.  We will use the term {\em parity} to denote the type
of an $\C{A}$-module: whether it is a left module or a right module.
We use the notation $C_{*\geq n}$ to denote the \emph{good truncation}
of $C_*$ at degree $n$, and $\Sigma C_*$ for the suspension of a
differential $\B{Z}$-graded $\Bbbk$-module $(C_*,d^C_*)$.

\subsection*{Acknowledgments}
We are grateful to the referee for pointing out the reference
\cite{PenaXi:HomologicalEpimorphisms} which allowed us to connect
H-unital ideals of Wodzicki~\cite{Wodzicki:Excision} and homological
epimorphisms of \cite{GeigleLenzing:PerpendicularCategories}.  We
would like to thank the referee also for pointing out the references
\cite{Happel:HochschildCohomologyOfFiniteDimensionalAlgebras,
  Cibils:TensorHochschild,
  MichelenaPlatzeck:HochschildCohomologyOfTriangularAlgebras} which
allowed us to better contextualize Conjecture~\ref{Conjecture}.  We
also thank Claude Cibils for pointing an error we made in the
published version of this preprint. The current version is corrected
by explicitly stating the Jacobi-Zariski sequence works only for
$p\geq 1$. An errata for the published version is to appear.

\section{Preliminaries}\label{Preliminaries}

\subsection{Relative (co)homology}
A monomorphism $f\colon X\to Y$ of $\C{A}$-modules is called an
$(\C{A},\C{B})$-monomorphism if $f$ is an monomorphism of
$\C{A}$-modules, and a split monomorphism of $\C{B}$-modules.  We
define $(\C{A},\C{B})$-epimorphisms similarly.  A short exact sequence
of $\C{A}$-modules $0\to X\xra{i} Y\xra{p} Z\to 0$ is called an
$(\C{A},\C{B})$-exact sequence if $i$ is an
$(\C{A},\C{B})$-monomorphism and $p$ is an
$(\C{A},\C{B})$-epimorphism.

An $\C{A}$-module $P$ is called an $(\C{A},\C{B})$-projective module
if for any $(\C{A},\C{B})$-epimorphism $f\colon X\to Y$ and a morphism
of $\C{A}$-modules $g\colon P \to Y$ there exists an $\C{A}$-module
morphism $g'\colon P\to X$ which satisfies $g = f\circ g'$
\[\xymatrix{
     & P \ar@{..>}[dl]_{g'} \ar[d]^g\\
X \ar[r]_f & Y \ar[r] & 0
}\]
$(\C{A},\C{B})$-injective modules are defined similarly.  Also, a
module $T$ is called $(\C{A},\C{B})$-flat if for every
$(\C{A},\C{B})$-short exact sequence $0\to X\to X'\to X''\to 0$ the
induced sequence of $\Bbbk$-modules
\[ 0\to X\otimes_\C{A} T\to X'\otimes_\C{A} T\to X''\otimes_\C{A} T\to
0 \] is exact.  Note that every $\C{A}$-projective (resp. $\C{A}$-flat
or $\C{A}$-injective) module is also $(\C{A},\C{B})$-projective
(resp. $(\C{A},\C{B})$-flat or $(\C{A},\C{B})$-injective) for any
unital subalgebra $\C{B}\subseteq \C{A}$.

\subsection{The bar complex}
The bar complex of a (unital) associative algebra $\C{A}$ is the
graded $\Bbbk$-module
\[ \CB_*(\C{A}) := \bigoplus_{n\geq 0} \C{A}^{\otimes n+2} \] together
with the differentials $d^\CB_n\colon \CB_n(\C{A})\to\CB_{n-1}(\C{A})$
which are defined by
\[ d^\CB_n(a_0\otimes\cdots\otimes a_{n+1})
   = \sum_{j=0}^{n} (-1)^j(\cdots\otimes a_ja_{j+1}\otimes\cdots)
\]
for any $n\geq 1$ and $a_0\otimes\cdots\otimes a_n\in \CB_n(\C{A})$.
The complex
\[ \CB_*(X;\C{A};Y) := X\otimes_\C{A} \CB_*(\C{A})\otimes_\C{A} Y \]
is called the two-sided (homological) bar complex of a pair $(X,Y)$ of
$\C{A}$-modules of opposite parity.  The cohomological counterpart
$\CB^*(X;\C{A};Z)$ for a pair of right $\C{A}$-modules $(X,Z)$ is
defined as
\[ \CB^*(X;\C{A};Z) := \Hom_\C{A}(\CB_*(X;\C{A};\C{A}),Z) \] The
corresponding bar complex for left modules is defined similarly using
$\CB_*(\C{A};\C{A};X)$.  Since the two-sided bar complex
$\CB_*(\C{A})$ is an $\C{A}^e$-projective resolution of the
$\C{A}$-bimodule $\C{A}$, the homology of the complex
$\CB_*(X;\C{A};Y)$ yields the $\Tor$-groups $\Tor^\C{A}_*(X,Y)$ for
the pair $(X,Y)$. Similarly, the $\Ext$-groups $\Ext_\C{A}^*(X,Z)$ for
the pair $(X,Z)$ come from the cohomological variant of the two-sided
bar complex $\CB^*(X;\C{A};Z)$.

To be technically correct, since our tensor products $\otimes$ are
taken over $\Bbbk$, the bar complexes we defined should be referred as
the {\em relative bar complexes} (relative to the base field) denoted
by $\CB_*(X;\C{A}|\Bbbk;Y)$ or $\CB^*(X;\C{A}|\Bbbk;Z)$.  However,
since $\Bbbk$ is a semi-simple subalgebra of $\C{A}$, the relative
(co)homology and the absolute (co)homology
agree~(cf. \cite[Prop.2.5]{Kaygun:CupProductII}). See also
\cite[Thm.1.2]{GerstenhaberSchack:RelativeHochschild}.

Now, we define the relative two-sided bar complexes
$\CB_*(X;\C{A}|\C{B};Y)$ and $\CB^*(X;\C{A}|\C{B};Z)$ similarly for
$\C{A}$-modules of the correct parity $X$, $Y$ and $Z$ where we
replace the tensor product $\otimes$ over $\Bbbk$ with the tensor
product $\otimes_\C{B}$ over a unital subalgebra $\C{B}$.  Since
$\C{A}\otimes_\C{B} \C{A}$ is $(\C{A},\C{B})$-projective by
\cite[Lem.2, pg. 248]{Hochschild:RelativeHomology}, we see that for
any right $\C{A}$-module $X$, the module $X\otimes_\C{B} \C{A}$ is a
$(\C{A},\C{B})$-projective module.  Then the relative complexes
$\CB_*(X;\C{A}|\C{B};Y)$ and $\CB^*(X;\C{A}|\C{B};Z)$ yield
respectively the relative torsion groups $\Tor^{(\C{A}|\C{B})}_*(X,Y)$
and the relative extension groups $\Ext_{(\C{A}|\C{B})}^*(X,Z)$.

\subsection{H-unital ideals}\label{HUnital}
A not necessarily unital $\Bbbk$-algebra $\C{I}$ is called {\em
  H-unital} if the bar complex $\CB_*(\C{I})$ of $\C{I}$ is a
resolution of $\C{I}$ viewed as a $\C{I}$-bimodule.  One can
immediately see that any unital algebra is H-unital.

Our definition of H-unitality differs from Wodzicki's original
definition given in \cite{Wodzicki:Excision}, but is still equivalent.
Wodzicki defines another complex $\CB'_*(\C{I})$ (he uses the notation
$B_*(\C{I})$)
\[ \CB'_*(\C{I}) = \bigoplus_{n\geq 0} \C{I}^{\otimes n+1} \]
together with the differentials
\[ d_n(x_0\otimes\cdots\otimes x_n)
   = \sum_{j=0}^{n-1} (-1)^j (\cdots\otimes x_j x_{j+1}\otimes\cdots)
\]
$\CB'_0(\C{I})$ is the differential graded $\Bbbk$-module
$\CB_*(\C{I})$ suspended once and augmented by $\C{I}$.  The fact that
$\CB_*(\C{I})$ is a resolution of $\C{I}$ is equivalent to the fact
that $\CB'_*(\C{I})$ is acyclic.  Because $\Bbbk$ is a field and every
$\Bbbk$-module is free, and therefore flat, the acyclicity of
$\CB'_*(\C{I})$ is equivalent to Wodzicki's definition of H-unitality
(cf.~\cite[pg.592]{Wodzicki:Excision}.)

\subsection{The Hochschild complex}\

The Hochschild complex of $\C{A}$ with coefficients in a
$\C{A}$-bimodule $M$ is the graded $\Bbbk$-module
\[ \CH_*(\C{A},M) := \bigoplus_{n\geq 0} M\otimes \C{A}^{\otimes n} \]
together with the differentials (traditionally denoted by $b$)
$b_n\colon \CH_n(\C{A},M)\to \CH_{n-1}(\C{A},M)$
\begin{align*}
  b_n(m\otimes a_1\otimes\cdots a_n)
  = & (ma_1\otimes a_2\otimes\cdots\otimes a_n)
      + \sum_{j=1}^{n-1} (-1)^j (m\otimes\cdots\otimes a_j a_{j+1}\otimes\cdots)\\
  & + (-1)^n (a_n m \otimes a_1\otimes\cdots\otimes a_{n-1})
\end{align*}
defined for $n\geq 1$ and $m\otimes a_1\otimes\cdots\otimes a_n\in
\CH_n(\C{A},M)$. One can define the relative Hochschild complex
$\CH_*(\C{A}|\C{B},M)$ similarly where we write the tensor
products over $\C{B}$ instead of $\Bbbk$.

For every $\C{A}$-bimodule $M$ the Hochschild complex $\CH_*(\C{A},M)$
can also be constructed as \[ \CH_*(\C{A},M) = M\otimes_{\C{A}^e}
\CB_*(\C{A}) \] where $\C{A}^e=\C{A}\otimes \C{A}^{op}$ is the
enveloping algebra of $\C{A}$.  Since the bar complex $\CB_*(\C{A})$
is a projective resolution of $\C{A}$ viewed as a $\C{A}$-bimodule,
the Hochschild homology groups we defined using the Hochschild complex
are $\Tor^{\C{A}^e}_n(\C{A},M)$
\cite[Prop.1.1.13]{Loday:CyclicHomology}.  One can similarly define
the Hochschild cohomology groups $HH^n(\C{A},M)$ as the extension
groups $\Ext_{\C{A}^e}^n(\C{A},M)$, which can also be computed as the
cohomology of the Hochschild cochain complex
$\CH^*(\C{A},M):=\Hom_{\C{A}^e}(\CB_*(\C{A}),M)$
\cite[Def.1.5.1]{Loday:CyclicHomology}.

We will use the notation $HH_*(\C{A})$ and $HH^*(\C{A})$ to denote
respectively $HH_*(\C{A},\C{A})$ and $HH^*(\C{A},\C{A})$ for a
(unital) $\Bbbk$-algebra $\C{A}$.

\subsection{Flat and r-flat extensions}\label{RFlat}
Assume $\C{B}\subseteq \C{A}$ is a unital subalgebra. This makes
$\C{A}$ into a $\C{B}$-bimodule.  Now, consider the following short
exact sequence of $\C{B}$-modules
\[ 0\to \C{B}\to \C{A}\to \C{A}/\C{B}\to 0 \]
Since $\C{B}$ is flat over itself, using the short exact sequence above
for every $\C{B}$-module $Y$ we get an exact sequence of $\Bbbk$-modules
of the form
\[ 0\to \Tor^\C{B}_1(\C{A},Y)\to
\Tor^\C{B}_1(\C{A}/\C{B},Y)\to \C{B}\otimes_\C{B}Y\to \C{A}\otimes_\C{B}Y
\to (\C{A}/\C{B})\otimes_\C{B}Y\to 0
\]
and an isomorphism of $\Bbbk$-modules $\Tor^\C{B}_n(\C{A},Y)\cong
\Tor^\C{B}_n(\C{A}/\C{B},Y)$ for every $n\geq 2$.  This means the flat
dimension (sometimes referred as the weak dimension) of $\C{A}/\C{B}$
is at most 1 when $\C{A}$ is a flat $\C{B}$-module.  On the other
hand, when the quotient $\C{A}/\C{B}$ is a flat $\C{B}$-module the
$\C{B}$-module $\C{A}$ must also be flat.
\begin{definition}\label{RFlatExtension}
  An inclusion of unital $\Bbbk$-algebras $\C{B}\subseteq\C{A}$ is
  called a {\em flat extension} if $\C{A}$ viewed as a
  $\C{B}$-bimodule is flat.  An inclusion of unital $\Bbbk$-algebras
  $\C{B}\subseteq\C{A}$ is called a {\em r-flat extension} if the
  quotient $\C{B}$-bimodule $\C{A}/\C{B}$ is $\C{B}$-flat.  A morphism
  of unital $\Bbbk$-algebras $\varphi\colon \C{B}\to\C{A}$ is called
  flat (resp. r-flat) if $im(\varphi)\subseteq \C{A}$ is a flat
  (resp. r-flat) extension.
\end{definition}

One can easily see that every r-flat extension is a flat extension.
Conversely, if $\C{A}$ is a flat extension over $\C{B}$ which is also
augmented, i.e. we have a unital algebra morphism
$\varepsilon\colon\C{A} \to \C{B}$ splitting the inclusion of algebras
$\C{B}\to \C{A}$, then it is also a r-flat extension.  In other words,
for augmented extensions flatness and r-flatness are equivalent.

\begin{example}
  The polynomial algebras $\C{B}[x_1,\ldots,x_n]$ with commuting
  indeterminates and the polynomial algebras $\C{B}\{x_1,\ldots,x_n\}$
  with non-commuting indeterminates are all r-flat extensions of a
  unital $\Bbbk$-algebra $\C{B}$.  In general, if $\C{G}$ is a monoid
  then the group algebra $\C{B}[\C{G}]$ of $\C{G}$ over $\C{B}$ gives
  us a r-flat extension $\C{B}\subseteq \C{B}[\C{G}]$.
\end{example}

\section{The Jacobi-Zariski Sequence for Torsion Groups}\label{JZT}

\begin{proposition}\label{Extension}
  Let $\C{I}$ be an H-unital ideal of a unital algebra $\C{B}$.
  Assume $X$ and $Y$ are $\C{B}/\C{I}$-modules of opposite parity.
  Then one has natural isomorphisms of the form
  \begin{equation}\label{TorIso}
    \Tor^\C{B}_p(X,Y)\cong \Tor^{\C{B}/\C{I}}_p(X,Y)
  \end{equation}
  for every $p\geq 0$.
\end{proposition}

\begin{proof}
  We begin by defining an auxiliary differential graded $\Bbbk$-module
  $I_*$. We augment $\CB'_*(\C{I})$ with $\Bbbk$ as follows
  \[ \Bbbk\xleftarrow{0}\C{I}\leftarrow\C{I}^{\otimes
    2}\leftarrow\C{I}^{\otimes 3}\leftarrow\cdots \] Since $\C{I}$ is
  H-unital, the homology is concentrated at degree $0$ where $H_0(I_*)
  = \Bbbk$ and $H_q(I_*)=0$ for every $q>0$. Then the product
  $X\otimes I_*\otimes Y$ is a differential graded $\Bbbk$-submodule
  of $\CB_*(X;\C{B};Y)$ since action of $\C{I}$ on $X$ and $Y$ are by
  $0$.  Now, we define an increasing filtration of differential graded
  $\Bbbk$-submodules $F^p_*\subseteq F^{p+1}_*$ of $\CB_*(X;\C{B};Y)$
  where for $0\leq p\leq n$ we let
  \[ F^p_n := \sum_{n_0+\cdots+n_p = n-p} 
     X\otimes \C{I}^{\otimes n_0}\otimes \C{B}\otimes \C{I}^{\otimes n_1}
      \otimes\cdots \otimes \C{B}\otimes \C{I}^{\otimes n_p}\otimes Y
  \]
  and allow $n_i=0$.  Then we define $F^p_n=F^{p+1}_n$ for every
  $p\geq n$, and let $F^{-1}_*:=0$. Note that $F^0_* = X\otimes
  I_*\otimes Y$ and $\bigcup_p F^p_*=\CB_*(X;\C{B};Y)$. In the
  associated spectral sequence of the filtration we get
  \[ E^1_{0,q} = H_q(F^0_*/F^{-1}_*) = H_q(X\otimes I_*\otimes Y)
  \cong X\otimes H_q(I_*)\otimes Y = 0 = E^\infty_{0,q} \] for every
  $q>0$ since $\C{I}$ is H-unital and $\otimes$ is an exact
  bifunctor. We also see that for $q<0$ \[ E^0_{p,q} =
  F^p_{p+q}/F^{p-1}_{p+q} = 0 = E^\infty_{p,q} \] because $F^p_n =
  F^{p+1}_n$ for every $p\geq n$.  Moreover, for any $p>0$ and for any
  $q\geq 0$
  \[ E^0_{p,q} = F^p_{p+q}/F^{p-1}_{p+q} \cong
  \bigoplus_{n_0+\cdots+n_p = q} X\otimes \C{I}^{\otimes
    n_0}\otimes(\C{B}/\C{I})\otimes \C{I}^{\otimes n_1} \otimes\cdots
  \otimes(\C{B}/\C{I})\otimes \C{I}^{\otimes n_p}\otimes Y
  \]
  Again, since the action of $\C{I}$ on $X$, $Y$ and $\C{B}/\C{I}$ are
  by 0, the quotient $F^p_*/F^{p-1}_*$ can be written as a product of
  differential graded modules
  \[ \Sigma^p(F^p_*/F^{p-1}_*) \cong X\otimes
  \underbrace{I_*\otimes(\C{B}/\C{I})\otimes\cdots\otimes
    I_*\otimes(\C{B}/\C{I})}_\text{$p$-times}\otimes I_*\otimes Y
  \]
  Our spectral sequence collapses because the homology of $I_*$ is
  concentrated at degree $0$. So, we get $E^1_{p,q}=0=E^\infty_{p,q}$
  for $q\neq 0$ and $p\geq 0$
  \[ E^1_{p,0} = H_p(F^p_*/F^{p-1}_*) = X\otimes(\C{B}/\C{I})^{\otimes
    p}\otimes Y \] which means $E^1_{*,0} = \CB_*(X;\C{B}/\C{I};Y)$.
  Hence we get $E^2_{p,0}=E^\infty_{p,0}=\Tor^{\C{B}/\C{I}}_p(X,Y)$.
  Our spectral sequence converges to the homology of
  $\CB_*(X;\C{B};Y)$ which is $\Tor^\C{B}_*(X,Y)$.  The result
  follows.
\end{proof}

\begin{remark}
  An epimorphism of algebras $\varphi\colon \C{B}\to \C{A}$ is called
  {\em a homological epimorphism} if the induced morphisms
  $\Ext_\C{A}^*(X,Y)\to \Ext_\C{B}^*(X,Y)$ is an isomorphism for every
  finitely generated $\C{A}$-modules $X$ and
  $Y$~\cite{GeigleLenzing:PerpendicularCategories}.  If $\C{B}$ is
  finite dimensional and we confine ourselves with finitely generated
  modules, then one has isomorphisms of the form~\eqref{TorIso} if and
  only if one has the same isomorphisms for Ext-groups. In other
  words, an epimorphism $\varphi\colon \C{B}\to\C{A}$ between two
  finite dimensional $\Bbbk$-algebras is a homology epimorphism when
  $ker(\varphi)$ is H-unital because of
  Proposition~\ref{Extension}. We would like to reiterate here again
  that we make no assumptions on the $\Bbbk$-dimensions of the
  algebras we work with, neither do we assume that our modules are
  finitely generated.
\end{remark}

\begin{theorem}\label{JZHomology}
  Assume $\C{B}$ is a unital $\Bbbk$-algebra and let $\C{B}\subseteq
  \C{A}$ be a r-flat extension.  Then $p\geq 1$ we have a long exact 
  sequence of the form
  \begin{equation}\label{LESHomology}
  \cdots\to \Tor^{(\C{A}|\C{B})}_{p+1}(X,Y)\to\Tor^\C{B}_p(X,Y)\to 
  \Tor^\C{A}_p(X,Y)\to \Tor^{(\C{A}|\C{B})}_p(X,Y)\to\cdots   
  \end{equation}
  for any right $\C{A}$-module $X$ and any left $\C{A}$-module $Y$
  where the last map
  $\Tor^\C{A}_1(X,Y)\to \Tor^{(\C{A}|\C{B})}_1(X,Y)$ in the sequence
  is an epimorphism.
\end{theorem}

\begin{proof}
  We consider the two-sided bar complex $\CB_*(X;\C{A};Y)$ for a
  right $\C{A}$-module $X$ and for a left $\C{A}$-module $Y$.  We 
  consider the following increasing filtration on $\CB_*(X;\C{A};Y)$
  \[ L^p_n = \sum_{n_0+\cdots+n_p = n-p} 
        X\otimes \C{B}^{\otimes n_0}\otimes \C{A}\otimes 
        \C{B}^{\otimes n_1}\otimes \cdots\otimes
        \C{A}\otimes \C{B}^{\otimes n_p}\otimes Y
  \] 
  for $p\leq n$ where we allow $n_i=0$.  We take $L^p_* = 0$ for $p<0$
  and let $L^p_n=L^{p+1}_n$ for $p\geq n$.  Note that the filtration
  degree comes from the number of tensor components which are equal to
  $\C{A}$.  We also see that $L^0_* = \CB_*(X;\C{B};Y)$ and $\colim_p
  L^p_* = \CB_*(X;\C{A};Y)$.  Then it is clear that $L^0_*/L^{-1}_* =
  L^0_* = \CB_*(X;\C{B};Y)$, and for $n\geq p$ we have
  \[ L^p_n/L^{p-1}_n \cong \bigoplus_{n_0+\cdots+n_p = n-p} 
       X\otimes\C{B}^{\otimes n_0}\otimes (\C{A}/\C{B})
       \otimes \C{B}^{\otimes n_1}\otimes\cdots\otimes 
       (\C{A}/\C{B})\otimes \C{B}^{\otimes n_p}\otimes Y
  \] 
  In the derived category, the quotient differential graded
  $\Bbbk$-module $L^p_*/L^{p-1}_*$ represents an $(p+2)$-fold product
  \[ X\otimes_\C{B}^L(\C{A}/\C{B})\otimes_\C{B}^L
  \cdots\otimes_\C{B}^L(\C{A}/\C{B}) \otimes_\C{B}^L Y \] where
  $\otimes_\C{B}^L$ denotes the left derived functor of
  $\otimes_\C{B}$.  We use the spectral sequence associated with this
  filtration employing the fact that $\C{B}\subseteq\C{A}$ is a r-flat
  extension and we see that the only non-zero groups are on the
  $p$-axis and on the $q$-axis
  \[ E^1_{p,q} = H_{p+q}(L^p_*/L^{p-1}_*) =
    \begin{cases}
    \Tor^\C{B}_q(X,Y) & \text{ if } p=0\\
    X\otimes_\C{B} \underbrace{(\C{A}/\C{B})\otimes_\C{B} \cdots\otimes_\C{B}
    (\C{A}/\C{B})}_\text{$p$-times}\otimes_\C{B} Y & \text{ if } q=0
    \end{cases}
  \]
  Now, we observe that the edge $q=0$ on the $E^1$-page is the
  normalized bar complex of $X$ and $Y$ over $\C{A}$ relative to
  $\C{B}$~\cite[Ch.VIII Thm.6.1 and Ch. IX
  Sec. 8-9]{MacLane:Homology}.  Then for $E^2_{p,q}$ we get
  \[ E^2_{p,q} = 
     \begin{cases}
       \Tor^\C{B}_q(X,Y) & \text{ if } p=0\\
       \Tor^{(\C{A}|\C{B})}_p(X,Y) & \text{ if } q=0
     \end{cases}
  \]
  For the subsequent pages in this spectral sequence, the only 
  relevant differentials are
  \[ d^p_{p,0}\colon \Tor^{(\C{A}|\C{B})}_p(X,Y)\to
    \Tor^\C{B}_{p-1}(X,Y) \] for $p\geq 2$, and the remaining terms
  for $p=0,1$ satisfy
  $\Tor^{(\C{A}|\C{B})}_p(X,Y) = E^2_{p,0} = E^\infty_{p,0}$.  Then
  for $p\geq 2$ we have short exact sequences of the form
  \begin{equation}\label{Part1}
    0\to E^\infty_{p,0}\to \Tor^{(\C{A}|\C{B})}_p(X,Y)\to \Tor^\C{B}_{p-1}(X,Y) \to
    E^\infty_{0,p-1} \to 0
  \end{equation}
  So, eventually in the $E^\infty$-page there remains only two classes 
  of non-zero groups: one on the $p$-axis the other on the $q$-axis.  
  The induced filtration 
  $L^* \Tor^\C{A}_*(X,Y)$ on the homology
  satisfies \[ E^\infty_{p,q} = L^p \Tor^\C{A}_{p+q}(X,Y)/L^{p-1}
  \Tor^\C{A}_{p+q}(X,Y)\] which can be placed into a diagram of the
  form
  \[\xymatrix{
    0\ \ar@{>->}[r] &
    L^0 \Tor^\C{A}_{p}(X,Y) \ \ar@{>->}[r] \ar@{->>}[d] & 
    L^1 \Tor^\C{A}_{p}(X,Y) \ \ar@{>->}[r] \ar@{->>}[d] & \cdots\ \ar@{>->}[r] &
    L^p \Tor^\C{A}_{p}(X,Y) \ \ar@{->>}[d] \\
    & E^\infty_{0,p} & E^\infty_{1,p-1} & & E^\infty_{p,0} 
  }\]
  where the top row is a sequence of injections with quotients given
  by the appropriate groups in $E^\infty$.  Since we only had non-zero
  groups on the $p$- and $q$-axes, we get short exact sequences of the
  form
  \begin{equation}\label{Part2}
    0 \to E^\infty_{0,p} \to \Tor^\C{A}_p(X,Y) \to E^\infty_{p,0}\to 0
  \end{equation}
  Splicing \eqref{Part1} and \eqref{Part2} we get a long exact
  sequence of the form~\eqref{LESHomology} for the range $p\geq 1$.
\end{proof}

\begin{corollary}
  Assume $\varphi\colon \C{B}\to \C{A}$ is a r-flat morphism such that
  $ker(\varphi)$ is an H-unital ideal of $\C{B}$.  Then there is a
  long exact sequence of the form~\eqref{LESHomology} for any right
  $\C{A}$-module $X$ and any left $\C{A}$-module $Y$.
\end{corollary}

\begin{proof}
  There is an overload of notations here.
  $\Tor^{(\C{A}|im(\varphi))}$ refers to the relative Tor with
  $im(\varphi)$ is a subalgebra of $\C{A}$. On the other hand
  $\Tor^{(\C{A}|\C{B})}$ refers to relative Tor using $\varphi\colon
  \C{B}\to \C{A}$ which, by definition, is constructed using the image
  of $\varphi$. In other words, if $\C{A}$ is viewed as a
  $\C{B}$-module via a morphism $\varphi\colon \C{B}\to \C{A}$ then
  $\C{A}\otimes_\C{B} \C{A} := \C{A}\otimes_{im(\varphi)}\C{A}$, and
  therefore $\Tor^{(\C{A}|im(\varphi))}$ and $\Tor^{(\C{A}|\C{B})}$
  are one and the same.  Then since $\varphi$ is r-flat, we have a
  long exact sequence similar to \eqref{LESHomology} of the form
  \[ \cdots\to \Tor^{(\C{A}|\C{B})}_{p+1}(X,Y)\to
     \Tor^{im(\varphi)}_p(X,Y)\to\Tor^\C{A}_p(X,Y)\to 
     \Tor^{(\C{A}|\C{B})}_p(X,Y)\to\cdots   
  \]
  for $p\geq 1$.  The result follows since we have isomorphisms of the
  form $\Tor_*^\C{B}(X,Y)\cong \Tor_*^{im(\varphi)}(X,Y)$ by
  Proposition~\ref{Extension}.
 \end{proof}

\section{The Jacobi-Zariski Sequence for Extension Groups}\label{JZE}

For the time being, we assume $\C{A}$ is a unital $\Bbbk$-algebra
which is also {\em finite dimensional} over $\Bbbk$, and $\C{B}$ is an
arbitrary unital subalgebra of $\C{A}$. We will work with the category
of finitely generated $\C{A}$-modules. In the category of finitely
generated $\C{A}$-modules of either parity, every module is
necessarily finite dimensional over $\Bbbk$. For these categories the
$\Bbbk$-duality functor $\Hom_\Bbbk(\ \cdot\ ,\Bbbk)$ gives us an
equivalence of categories of the form \[ \Hom_\Bbbk(\ \cdot\
,\Bbbk)\colon \lmod{\C{A}}\to \rmod{\C{A}} \] where $\lmod{\C{A}}$ and
$\rmod{\C{A}}$ denote the categories of finitely generated left and
right $\C{A}$-modules, respectively. Using the adjunction between the
functors $\otimes_\C{A}$ and $\Hom_\Bbbk(\ \cdot\ ,\Bbbk)$ one can
easily prove the following lemma.

\begin{lemma}\label{FlatInjective}
  Assume $\C{A}$ is a finite dimensional $\Bbbk$-algebra and let $Y$
  be a finitely generated $\C{A}$-module.  Then $Y$ is
  $(\C{A},\C{B})$-flat if and only if the $\Bbbk$-dual
  $\Hom_\Bbbk(Y,\Bbbk)$ is an $(\C{A},\C{B})$-injective module.
\end{lemma}
Note that the absolute case is covered by $\C{B}=\Bbbk$.
\begin{proof}
  Assume $Y$ is $(\C{A},\C{B})$-flat, that is every
  $(\C{A},\C{B})$-exact sequence of the form $0\to X\to X'\to X''\to
  0$ induce an exact sequence of $\Bbbk$-modules the form
  \[ 0\to X\otimes_\C{A} Y\to X'\otimes_\C{A} Y\to X''\otimes_\C{A}
  Y\to 0 \] Since $\Bbbk$ is a field, and therefore $\Hom_\Bbbk(\
  \cdot\ ,\Bbbk)$ is an exact functor, we get another exact sequence
  of $\Bbbk$-modules of the form
  \[ 0\to \Hom_\Bbbk(X\otimes_\C{A} Y,\Bbbk)\to
  \Hom_\Bbbk(X'\otimes_\C{A} Y,\Bbbk)\to \Hom_\Bbbk(X''\otimes_\C{A}
  Y,\Bbbk)\to 0 \] Using the adjunction between $\otimes_\C{A}$ and
  $\Hom_\Bbbk(\ \cdot\ ,\Bbbk )$ we get
  \[ 0\to \Hom_\C{A}(X,\Hom_\Bbbk (Y,\Bbbk ))\to
  \Hom_\C{A}(X',\Hom_\Bbbk (Y,\Bbbk ))\to \Hom_\C{A}(X'',\Hom_\Bbbk
  (Y,\Bbbk ))\to 0 \] is also exact.  This is equivalent to the fact
  that $\Hom_\Bbbk (Y,\Bbbk )$ is $(\C{A},\C{B})$-injective.
  Conversely, assume $Y$ is $(\C{A},\C{B})$-injective.  Since
  $\Hom_\Bbbk (\ \cdot\ ,\Bbbk )$ is an equivalence of finitely
  generated $\C{A}$-modules exchanging parity, there is a finitely
  generated $\C{A}$-module $T$ of the opposite parity of $Y$ such that
  $Y\cong \Hom_\Bbbk (T,\Bbbk )$.  Moreover, one can easily see that
  $\Hom_\Bbbk (Y,\Bbbk )\cong T$ for the same module since the modules
  we consider are all finite dimensional.  Assume now, without any
  loss of generality, that $Y$ is a right module, and therefore $T$ is
  left module.  Then
  \[ 0\to \Hom_\C{A}(X,\Hom_\Bbbk (T,\Bbbk ))\to \Hom_\C{A}(X',\Hom_\Bbbk (T,\Bbbk ))\to
  \Hom_\C{A}(X'',\Hom_\Bbbk (T,\Bbbk ))\to 0 \] is exact for every
  $(\C{A},\C{B})$-exact sequence $0\to X\to X'\to X''\to 0$ of
  finitely generated $\C{A}$-modules.  Using the adjunction again we
  see that
  \[ 0\to \Hom_\Bbbk (X\otimes_\C{A} T,\Bbbk )\to \Hom_\Bbbk
  (X'\otimes_\C{A} T,\Bbbk )\to\Hom_\Bbbk (X''\otimes_\C{A} T,\Bbbk
  )\to 0\] is also $(\C{A},\C{B})$-exact.  Since $\Hom_\Bbbk (\ \cdot\
  ,\Bbbk )$ is an exact equivalence for finite dimensional
  $\Bbbk$-modules we see that
  \[ 0\to X\otimes_\C{A} T\to X'\otimes_\C{A} T\to X''\otimes_\C{A}
  T\to 0\] is also $(\C{A},\C{B})$-exact.  This is equivalent to the
  fact that $T$ is $(\C{A},\C{B})$-flat.
\end{proof}

\begin{lemma}\label{JZCohomology}
  If $\C{A}$ is a finite dimensional $\Bbbk$-algebra and
  $\C{B}\subseteq \C{A}$ is a r-flat extension then for any finitely
  generated $\C{A}$-modules $X$ and $Y$ we have a long exact sequence
  in cohomology of the form
  \begin{equation}\label{LESCohomology}
    \cdots\to \Ext_{(\C{A}|\C{B})}^p(X,Y) \to \Ext_{\C{A}}^p(X,Y) \to
    \Ext_{\C{B}}^p(X,Y) \to \Ext_{(\C{A}|\C{B})}^{p+1}(X,Y) \to \cdots
  \end{equation}
  for every $p\geq 1$ where the first map
  $\Ext_{(\C{A}|\C{B})}^1(X,Y) \to \Ext_{\C{A}}^1(X,Y)$ in the
  sequence is a monomorphism.
\end{lemma}

\begin{proof}
  Assume $X$ and $Y$ are $\C{A}$-modules.  Let $\C{U}$ and $\C{V}$ be
  a unital $\Bbbk$-subalgebras of $\C{A}$ such that
  $\C{V}\subseteq\C{U}$.  Fix a pre-dual $T$ of $Y$, i.e. $\Hom_\Bbbk
  (T,\Bbbk )\cong Y$.  Now, observe that if $T$ is left $\C{A}$-module
  then $\CB_*(\C{U};\C{U}|\C{V};T)$ is a $(\C{U},\C{V})$-projective
  resolution of $T$, and therefore $\Hom_\Bbbk
  (\CB_*(\C{U};\C{U}|\C{V};T),\Bbbk )$ is an injective resolution of
  $Y \cong \Hom_\Bbbk (T,\Bbbk )$ by Lemma~\ref{FlatInjective}.  Then
  \begin{align*}
    \Ext_{(\C{U}|\C{V})}^n(X,Y) 
    \cong & \Ext_{(\C{U}|\C{V})}^n(X,\Hom_\Bbbk (T,\Bbbk )) \\
    = & H_n \Hom_\C{U}(X,\Hom_\Bbbk (\CB_*(\C{U};\C{U}|\C{V};T),\Bbbk ))\\
    \cong & H_n \Hom_\Bbbk (\CB_*(X;\C{U}|\C{V};T),\Bbbk ) \\
    \cong & \Hom_\Bbbk (\Tor^{(\C{U}|\C{V})}_n(X,T),\Bbbk )
  \end{align*}
  We consider three cases now: (i) $\C{U}=\C{A}$ and $\C{V}=\Bbbk$,
  (ii) $\C{U}=\C{B}$ and $\C{V}=\Bbbk$ and finally (iii) $\C{U}=\C{A}$
  and $\C{V}=\C{B}$.  Then we see that if we apply $\Hom_\Bbbk (\
  \cdot\ ,\Bbbk )$ to the long exact sequence~\eqref{LESHomology} we
  obtain the result we would like to prove.
\end{proof}

\begin{remark}\label{FiniteExtension}
  Note that in the proof of Lemma~\ref{FlatInjective} where we show
  $\Hom_\Bbbk (Y,\Bbbk )$ is $(\C{A},\C{B})$-injective when it is
  $(\C{A},\C{B})$-flat, we did not use the fact that $Y$ is finitely
  generated over a finite dimensional $\Bbbk$-algebra $\C{A}$.
  Indeed, it is true that for an arbitrary $\C{A}$-module over any
  $\Bbbk$-algebra $\C{A}$ (regardless of $\Bbbk$-dimension) and for an
  arbitrary unital subalgebra $\C{B}$ the module $\Hom_\Bbbk (Y,\Bbbk
  )$ is still $(\C{A},\C{B})$-injective when $Y$ is
  $(\C{A},\C{B})$-flat.  Then the result of Lemma~\ref{JZCohomology}
  can be extended to not necessarily finite dimensional algebras, if
  instead of the full category of finitely generated $\C{A}$-modules
  we consider the category of $\C{A}$-modules which are {\em finite
    dimensional over the base field} $\Bbbk$.  The result we proved in
  Lemma~\ref{FlatInjective} in the direction we need is still true and
  in the category of finite dimensional $\C{A}$-modules, the duality
  functor is still an equivalence.  These are enough to prove an
  analogue of Lemma~\ref{JZCohomology} in the case $\C{A}$ is not
  necessarily finite dimensional over $\Bbbk$.
\end{remark}

\begin{definition}
  Assume $\C{A}$ is an arbitrary $\Bbbk$-algebra, and we make no
  assumption on the $\Bbbk$-dimension of $\C{A}$.  An $\C{A}$-module
  $M$ is called an {\em approximately finite dimensional}
  $\C{A}$-module if $M$ is a colimit of all of its $\C{A}$-submodules
  of finite $\Bbbk$-dimension.
\end{definition}

\begin{theorem}\label{RealJZCohomology}
  Let $\C{A}$ be an arbitrary unital $\Bbbk$-algebra where we make no
  assumption on the $\Bbbk$-dimension of $\C{A}$.  Let $\C{B}\subseteq
  \C{A}$ be a r-flat extension and assume $Y$ is an approximately
  finite dimensional $\C{A}$-module.  Then for every $\C{A}$-module
  $X$ we have a long exact sequence of the form~\eqref{LESCohomology}.
\end{theorem}

\begin{proof}
  We know by Remark~\ref{FiniteExtension} that we have a long exact
  sequence of the form
  \[  \cdots\to \Ext_{(\C{A}|\C{B})}^p(X,Z) \to \Ext_{\C{A}}^p(X,Z) \to
      \Ext_{\C{B}}^p(X,Z) \to \Ext_{(\C{A}|\C{B})}^{p+1}(X,Z) \to \cdots
  \]
  for every submodule $Z$ of $Y$ of finite $\Bbbk$-dimension for every
  $p\geq 1$.  Then we use the fact $Y$ is the colimit of all of its
  $\C{A}$-submodules of finite $\Bbbk$-dimension and that colimit is
  an exact functor.
\end{proof}

\begin{corollary}
  Assume $\varphi\colon \C{B}\to \C{A}$ is a r-flat morphism of unital
  algebras such that $ker(\varphi)$ is an H-unital ideal of
  $\C{B}$. Assume $Y$ is an approximately finite dimensional
  $\C{A}$-module.  Then for every $\C{A}$-module $X$ we have a long
  exact sequence of the form~\eqref{LESCohomology}.
\end{corollary}

\section{Jacobi-Zariski Sequence for Hochschild Homology and Cyclic (Co)Homology}\label{MainResult}

In this section we assume $\C{A}$ is a unital $\Bbbk$-algebra and
$\C{B}\subseteq \C{A}$ is a r-flat extension.  We make no assumption
on the $\Bbbk$-dimension of $\C{A}$.  Our aim in this section is to
repeat the argument we gave in Section~\ref{JZT} and Section~\ref{JZE}
for the Hochschild homology and cyclic (co)homology, and prove
appropriate versions of Theorem~\ref{JZHomology} and
Theorem~\ref{RealJZCohomology}.  Even though the filtrations are
similar, the associated graded complexes will be different.  So, we
need to check the details carefully.

Now we define an increasing filtration on the Hochschild chain complex
by letting
\begin{equation}\label{JZHFiltration}
 G^p_n = \sum_{n_0+\cdots+n_p = n-p} M\otimes
 \C{B}^{\otimes n_0}\otimes \C{A}\otimes \C{B}^{\otimes n_1}\otimes\cdots\otimes
 \C{A}\otimes \C{B}^{\otimes n_p}
\end{equation}
for $0\leq p\leq n$. We let $G^p_n=0$ for $p<0\leq n$ and $G^p_n =
G^{p+1}_n$ for all $p\geq n\geq 0$.  Here again, the filtration degree
counts the number of tensor components which are equal to $\C{A}$ and
observe that $G^0_* = \CH_*(\C{B},M)$.  Moreover, we see that
$\colim_p G^p_* = \CH_*(\C{A},M)$.  The associated graded complex is
then given by
\[ G^p_n/G^{p-1}_n \cong \bigoplus_{n_0+\cdots+n_p = n-p} M\otimes
\C{B}^{\otimes n_0}\otimes (\C{A}/\C{B})\otimes \C{B}^{\otimes
  n_1}\otimes\cdots\otimes (\C{A}/\C{B})\otimes \C{B}^{\otimes n_p} \]
for any $n\geq p\geq 1$, and at degree $0$ we see $G^0_*/G^{-1}_* =
G^0_* = \CH_*(\C{B},M)$.  But recall that $\C{A}/\C{B}$ is flat over
$\C{B}$ since $\C{B}\subseteq\C{A}$ is a r-flat extension.  Then on
the $E^1$-page of the associated spectral sequence we get only two
non-zero groups: one on the $p$-axis and the other on the $q$-axis
\begin{equation*}
  E^1_{p,q} = H_{p+q}(G^p_*/G^{p-1}_*) = 
  \begin{cases}
    HH_q(\C{B},M) & \text{ if } p=0\\
    \widetilde{\CH}_p(\C{A}|\C{B},M) & \text{ if } q=0 \text{ and } p>0
  \end{cases}
\end{equation*}
where $\widetilde{\CH}_*(\C{A}|\C{B},M)$ is the {\em normalized}
relative Hochschild complex~\cite[1.1.14]{Loday:CyclicHomology}.  This leads us
to the $E^2$-page
\begin{equation*}
  E^2_{p,q} = 
  \begin{cases}
    HH_q(\C{B},M) & \text{ if } p=0 \text{ and } q>0\\
    HH_p(\C{A}|\C{B},M) & \text{ if } q=0
  \end{cases}
\end{equation*}
The rest of the argument is similar to the argument we gave in
Section~\ref{JZT}.

\begin{theorem}\label{JacobiZariskiHochschild}
  Assume $\C{B}\subseteq \C{A}$ is a r-flat extension of unital,
  associative but not necessarily commutative $\Bbbk$-algebras.  Then
  for any $\C{A}$-bimodule $M$ we have a long exact sequence of the
  form
  \[ \cdots\to HH_{p+1}(\C{A}|\C{B},M)\to HH_p(\C{B},M)\to
    HH_p(\C{A},M)\to HH_p(\C{A}|\C{B},M)\to \cdots\] for $p\geq 1$
  where the last map $ HH_1(\C{A},M)\to HH_1(\C{A}|\C{B},M)$ in the
  sequence is an epimorphism.  If we assume $M$ is approximately
  finite dimensional then we obtain have a long exact sequence in
  cohomology of the form
  \[ \cdots \to HH^p(\C{A}|\C{B},M)\to HH^p(\C{A},M)\to
    HH^p(\C{B},M)\to HH^{p+1}(\C{A}|\C{B},M)\to\cdots \] for every
  $p\geq 1$ where the first map $HH^1(\C{A}|\C{B},M)\to HH^1(\C{A},M)$
  in the sequence is a monomorphism.
\end{theorem}

Now we are ready to extend this result to cyclic (co)homology.
\begin{theorem}\label{Theorem:JZW}
  Assume $\C{A}$, $\C{B}$ and $\Bbbk$ are as before. Then we have the
  following long exact sequences:
  \begin{align}
    \cdots\to HH_{p+1}(\C{A}|\C{B})\to HH_p(\C{B})\to HH_p(\C{A})\to HH_p(\C{A}|\C{B})\to \cdots\label{JZPureHochschild}\\
    \cdots\to HC_{p+1}(\C{A}|\C{B})\to HC_p(\C{B})\to HC_p(\C{A})\to HC_p(\C{A}|\C{B})\to \cdots\\
    \cdots\to HC^p(\C{A}|\C{B})\to HC^p(\C{A})\to HC^p(\C{B})\to
    HC^{p+1}(\C{A}|\C{B})\to \cdots
  \end{align}
  for $p\geq 1$.
\end{theorem}

\begin{proof}
  In order to prove our assertion for cyclic (co)homology, we first
  need to derive~\eqref{JZPureHochschild} which is a version of the
  Jacobi-Zariski sequence for Hochschild homology.  For this purpose,
  we will need to write a filtration similar to the filtration we
  defined in \eqref{JZHFiltration}.  So, for $0\leq p\leq n+1$ we
  define
  \[ G^p_n = \sum_{n_0+\cdots+n_p = n+1-p} \C{B}^{\otimes n_0}\otimes
  \C{A}\otimes \C{B}^{\otimes n_1}\otimes\cdots\otimes \C{A}\otimes
  \C{B}^{\otimes n_p} \] We let $G^p_n = G^{p+1}_n$ for all $p\geq
  n+1$ and $G^p_n=0$ for $p<0$.  Observe that the filtration is
  compatible with the actions of the cyclic groups $\B{Z}/(n+1)$ at
  every degree $n\geq 0$.  The rest of the proof for obtaining the
  long exact sequence~\eqref{JZPureHochschild} is similar to argument
  we used at the beginning of this section for Hochschild homology
  with coefficients and the proof of Theorem~\ref{JZHomology}.  The
  proof for cyclic (co)homology also follows after our observation
  that the filtration $G^p_*$ is compatible with the actions of the
  cyclic groups.
\end{proof}

Let us use $C_{*\geq n}$ for the \emph{good
  truncation}~\cite[1.2.7]{Weibel:HomologicalAlgebra} of a
differential graded $\Bbbk$-module $C_*$.  Then we have the following:
\begin{theorem}\label{AllInOne}
  Let $\varphi\colon \C{B}\to\C{A}$ be a r-flat morphism of unital
  $\Bbbk$-algebras such that $\C{I}:=ker(\varphi)$ is H-unital.  Let
  $\CH_*(\C{A},\C{B})$ be the homotopy cofiber of the morphism
  $\varphi_*\colon \CH_*(\C{B})\to\CH_*(\C{A})$.  Then there is a
  homotopy cofibration sequence of the form
  \begin{equation}\label{JZSW}
    \Sigma\CH_{*\geq 1}(\C{I})\to \CH_{*\geq 1}(\C{A},\C{B})\to \CH_{*\geq 1}(\C{A}|\C{B}) 
  \end{equation}
  which induces a long exact sequence of the form
  \[\cdots\to HH_{p+2}(\C{A}|\C{B})\to HH_p(\C{I})\to HH_{p+1}(\C{A},\C{B})\to 
    HH_{p+1}(\C{A}|\C{B})\to\cdots\] for $p\geq 1$, and we get an
  isomorphism of the form $HH_1(\C{A},\C{B})\cong HH_1(\C{A}|\C{B})$.
  Analogous sequences exist for the cyclic homology and cohomology
  with no additional hypothesis.
\end{theorem}

\begin{proof}
  We will give the proof for the Hochschild homology.  The proofs for
  cyclic homology and cohomology are similar, and therefore, omitted.
  We consider all of our Hochschild chain complexes as differential
  graded $\Bbbk$-modules inside (bounded below) derived category of
  $\Bbbk$-modules.  This is a triangulated category.  In this category
  we have two distinguished triangles: one coming from the Wodzicki
  excision sequence~\cite[Thm.3.1]{Wodzicki:Excision}
  \[ \CH_*(\C{I})\to \CH_*(\C{B})\xra{\pi_*} \CH_*(\C{B}/\C{I})\to
    \Sigma\CH_*(\C{I})\]
  and the other coming from our Jacobi-Zariski sequence:
  \[ \CH_{*\geq 1}(\C{B}/\C{I})\xra{i_*} \CH_{*\geq 1}(\C{A})\to
    \CH_{*\geq 1}(\C{A}|\C{B})\to \Sigma\CH_{*\geq 1}(\C{B}/\C{I}) \]
  where we use the \emph{good truncation} of the Hochschild complexes
  because our Jacobi-Zariski sequence works only for the range
  $p\geq 1$.  Now, consider another homotopy cofibration sequence
  $\CH_{*\geq 1}(\C{B})\xra{\varphi_*} \CH_{*\geq 1}(\C{A})\to
  \CH_{*\geq 1}(\C{A},\C{B})$ where the last term is defined as the
  homotopy cofiber of the morphism
  $\varphi_*\colon\CH_{*\geq 1}(\C{B})\to\CH_{*\geq 1}(\C{A})$.  We
  construct
  \[\xymatrix{
      \CH_{*\geq 1}(\C{B})\ar[r]^{\pi_*}\ar@{=}[d] & 
      \CH_{*\geq 1}(\C{B}/\C{I})\ar[r] \ar[d] & 
      \Sigma\CH_{*\geq 1}(\C{I})\ar[r] \ar@{..>}[d]^{\gamma_*}& 
      \Sigma\CH_{*\geq 1}(\C{B})\ar@{=}[d]\\  
      \CH_{*\geq 1}(\C{B})\ar[r]_{\varphi_*} & 
      \CH_{*\geq 1}(\C{A})\ar[r]  &
      \CH_{*\geq 1}(\C{A},\C{B})\ar[r] & 
      \Sigma\CH_{*\geq 1}(\C{B})
    }\]
  Now, $\gamma_*$ exists because the (bounded below) derived category
  of $\Bbbk$-modules is a triangulated category.  We also have that
  the middle square is homotopy
  cartesian~\cite[Lem.1.4.3]{Neeman:TriangulatedCategories}.  Thus we
  get a distinguished triangle of the form
  \[ \Sigma\CH_{*\geq 1}(\C{I})\xra{\gamma_*} \CH_{*\geq
      1}(\C{A},\C{B})\to \CH_{*\geq 1}(\C{A}|\C{B})\to
    \Sigma^2\CH_{*\geq 1}(\C{I}) \] This follows from
  Theorem~\ref{Theorem:JZW} since $\CH_{*\geq 1}(\C{A}|\C{B})$ is the
  homotopy cofiber of the morphism
  $\CH_{*\geq 1}(\C{B}/\C{I})\to \CH_{*\geq 1}(\C{A})$ induced by the
  inclusion $im(\varphi)\subseteq \C{A}$.  This finishes the proof.
\end{proof}

One can easily see that Theorem~\ref{AllInOne} gives us the
Jacobi-Zariski sequence when $\varphi$ is monomorphism with an r-flat
image and the Wodzicki excision sequence when $\varphi$ is an
epimorphism with an H-unital kernel, for Hochschild homology and
cyclic (co)homology.

\begin{conjecture}\label{Conjecture}
  We believe that an appropriate analogue of Theorem~\ref{Theorem:JZW}
  holds for Hochschild cohomology when $p$ is large enough.  However,
  due to its non-functoriality, to prove such an analogue for
  Hochschild cohomology is not a straightforward task.  If
  $\C{B}\subseteq \C{A}$ is an arbitrary extension, we have a homotopy
  cofibration sequence of the form $\CH^*(\C{B})\to
  \CH^*(\C{B},\C{A})\to \CH^*(\C{B},\C{A}/\C{B})$ coming from the
  short exact sequence of $\C{B}$-bimodules $0\to \C{B}\to \C{A}\to
  \C{A}/\C{B}\to 0$. The last term $\CH^*(\C{B},\C{A}/\C{B})$
  calculates $HH^*(\C{B},\C{A}/\C{B})$ which is
  $\Ext_{\C{B}^e}^*(\C{B},\C{A}/\C{B})$. So, we have $HH^n(\C{B})\cong
  HH^n(\C{B},\C{A})$ for $n\geq 2$ when $\C{A}/\C{B}$ is an injective
  $\C{B}^e$-module.  If $\C{A}$ is approximately finite dimensional
  and $\C{A}/\C{B}$ is $\C{B}^e$-flat on top of being
  $\C{B}^e$-injective, we can use
  Theorem~\ref{JacobiZariskiHochschild} where we set $M=\C{A}$ to
  obtain an analogue of Theorem~\ref{Theorem:JZW} for Hochschild
  cohomology but for $p\geq 2$.  There are similar results pointing in
  this direction. See for example
  \cite[Thm.5.3]{Happel:HochschildCohomologyOfFiniteDimensionalAlgebras},
  \cite[Thm.4.5]{Cibils:TensorHochschild} and
  \cite{MichelenaPlatzeck:HochschildCohomologyOfTriangularAlgebras}.
\end{conjecture}

\end{document}